\documentclass[11pt]{amsart}
\usepackage[utf8]{inputenc}
\usepackage{enumerate,amssymb,xcolor,comment}
\usepackage[a4paper, footskip=1cm, headheight = 16pt, top=3.7cm, bottom=3cm,  right=2.5cm,  left=2.5cm ]{geometry}
\usepackage[colorlinks,linkcolor=blue,citecolor=red]{hyperref}
\usepackage[center]{caption}

\newtheorem{question}{Question}[section]
\newtheorem{lemma}[question]{Lemma}
\newtheorem{theorem}[question]{Theorem}

\usepackage[foot]{amsaddr}

\usepackage[T1]{fontenc}

\usepackage[leqno]{amsmath}

\makeatletter
\newcommand{\leqnomode}{\tagsleft@true}
\newcommand{\reqnomode}{\tagsleft@false}
\makeatother

\usepackage{latexsym}
\usepackage{amsfonts}

\usepackage{tikz}
\usepackage{float}
\usepackage{lmodern}

\def\dd{\hbox{-}}

\usepackage{marvosym}               

\DeclareMathOperator{\tw}{tw}

\newcounter{tbox}

\newcommand{\sta}[1]{\vspace*{0.3cm}\refstepcounter{tbox}\noindent{ \parbox{\textwidth}{(\thetbox) \emph{#1}}}\vspace*{0.3cm}}

\newcommand{\mylongtitle}[1]{%
  \ifodd\value{page}%
    \protect\parbox{0.97\linewidth}{#1}\hfill%
  \else%
    \hfill\protect\parbox{0.97\linewidth}{#1}%
  \fi%
}

\usepackage{latexsym}
\usepackage{amsfonts}
\usepackage[leqno]{amsmath}
\usepackage{tikz}
\usepackage{float}
\usepackage{lmodern}
\usepackage[T1]{fontenc}

\def\dd{\hbox{-}}

\usepackage{marvosym}

\newcommand{\otherlabel}[2]{\protected@edef\@currentlabel{#2}\label{#1}}
\makeatother

\mathchardef\mh="2D

\title[Induced subgraphs and tree decompositions V.]{Induced subgraphs and tree decompositions\\
V. One neighbor in a hole}

\author{Tara Abrishami$^{\ast \dagger}$}
\author{Bogdan Alecu$^{\ast \ast \mathparagraph}$}
\author{Maria Chudnovsky$^{\ast \amalg}$}
\author{Sepehr Hajebi $^{\mathsection}$}
\author{Sophie Spirkl$^{\mathsection \parallel}$}
\author{Kristina Vu\v{s}kovi\'{c}$^{\ast \ast \ddagger}$}

\address{$^{\ast}$Princeton University, Princeton, NJ, USA}
\address{$^{**}$School of Computing, University of Leeds, Leeds, UK}
\address{$^{\mathsection}$Department of Combinatorics and Optimization, University of Waterloo, Waterloo, Ontario, Canada}
\address{$^{\dagger}$ Supported by NSF-EPSRC Grant DMS-2120644.}
\address{$^{\amalg}$ Supported by NSF-EPSRC Grant DMS-2120644 and by AFOSR grant FA9550-22-1-0083.} 
     \address{$^{\mathparagraph}$ Supported by DMS-EPSRC Grant EP/V002813/1.} 
\address{$^{\parallel}$ We acknowledge the support of the Natural Sciences and Engineering Research Council of Canada (NSERC), [funding reference number RGPIN-2020-03912].
Cette recherche a \'et\'e financ\'ee par le Conseil de recherches en sciences naturelles et en g\'enie du Canada (CRSNG), [num\'ero de r\'ef\'erence RGPIN-2020-03912]. This project was funded in part by the Government of Ontario.}
\address{$^{\ddagger}$ Partially supported by DMS-EPSRC Grant EP/V002813/1.}

\date {\today}

\allowdisplaybreaks

\begin{document}
\raggedbottom

\maketitle
\begin{abstract}What are the unavoidable induced subgraphs of graphs with large treewidth? It is well-known that the answer must include a complete graph, a complete bipartite graph, all subdivisions of a wall and line graphs of all subdivisions of a wall (we refer to these graphs as the “basic treewidth obstructions”). So it is natural to ask whether graphs excluding the basic treewidth obstructions as induced subgraphs have bounded treewidth. Sintiari and Trotignon answered this question in the negative. Their counterexamples, the so-called ``layered wheels,'' contain wheels, where a \textit{wheel} consists of a \textit{hole} (i.e., an induced cycle of length at least four) along with a vertex with at least three neighbors in the hole.  This leads one to ask  whether graphs excluding wheels and the basic treewidth obstructions as induced subgraphs have bounded treewidth. This also turns out to be false due to Davies’ recent example of graphs with large treewidth,  no wheels  and no  basic treewidth obstructions as induced subgraphs. However, in Davies' example there exist holes and vertices (outside of the hole) with two neighbors in them. Here we prove that a hole with a vertex with at least two neighbors in it is inevitable in graphs with large treewidth and no basic obstruction. Our main result is that  graphs in which every vertex has at most one neighbor in every hole (that does not contain it) and with the basic treewidth obstructions excluded as induced subgraphs have bounded treewidth.

\end{abstract}

\section{Introduction}
All graphs in this paper are finite and simple. 
Let $H$ and $G$ be graphs.
We say $G$ \emph{contains} $H$ if $G$ has an induced subgraph isomorphic to $H$ (unless stated otherwise). We say that $G$ is \emph{$H$-free} if $G$ does not contain $H$. For a family of graphs $\mathcal{H}$, we say that $G$ is $\mathcal{H}$-free if $G$ is $H$-free for
every $H \in \mathcal{H}$.
A \emph{tree decomposition} $(T, \chi)$ of $G$ consists of a tree $T$ and a map $\chi: V(T) \to 2^{V(G)}$ such that the following hold: 
\begin{enumerate}[(i)]
\itemsep -.2em
    \item For every vertex $v \in V(G)$, there exists $t \in V(T)$ such that $v \in \chi(t)$. 
    
    \item For every edge $v_1v_2 \in E(G)$, there exists $t \in V(T)$ such that $v_1, v_2 \in \chi(t)$.
    
    \item For every $v \in V(G)$, the subgraph of $T$ induced by $\{t \in V(T) \mid v \in \chi(t)\}$ is connected.
\end{enumerate}

If $(T, \chi)$ is a tree decomposition of $G$ and $V(T) = \{t_1, \hdots, t_n\}$, the sets $\chi(t_1), \hdots, \chi(t_n)$ are called the \emph{bags of $(T, \chi)$}.  The \emph{width} of a tree decomposition $(T, \chi)$ is $\max_{t \in V(T)} |\chi(t)|-1$. The \emph{treewidth} of $G$, denoted $\tw(G)$, is the minimum width of a tree decomposition of $G$.

Treewidth is an extensively-studied graph parameter, mostly due to the fact that graphs of bounded treewidth exhibit interesting structural
\cite{RS-GMXVI} and algorithmic \cite{Bodlaender1988DynamicTreewidth} properties. It is thus of interest to understand the  unavoidable substructures emerging in graphs of large treewidth (these are often referred to as ``obstructions to bounded treewidth''). For instance, for each $k$, the {\em $(k \times k)$-wall}, denoted by $W_{k \times k}$, is a planar graph with maximum degree three and with treewidth $k$ (see Figure \ref{fig:5x5wall}; a precise definition can be found in \cite{wallpaper}). Every subdivision of $W_{k \times k}$ is also a
graph of treewidth $k$. The unavoidable subgraphs of graphs with large treewidth are fully characterized by the Grid Theorem of Robertson and Seymour, the following.
\begin{theorem}[\cite{RS-GMV}]\label{wallminor}
There is a function $f: \mathbb{N} \rightarrow \mathbb{N}$
such that every graph of treewidth at least $f(k)$
contains a subdivision of $W_{k \times k}$ as a subgraph.
\end{theorem}

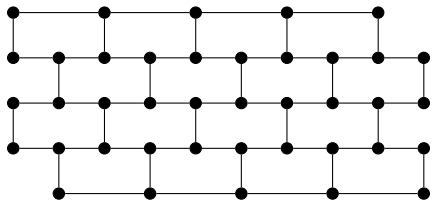
\begin{figure}[!htb]
\centering

\begin{tikzpicture}[scale=2,auto=left]
\tikzstyle{every node}=[inner sep=1.5pt, fill=black,circle,draw]  
\centering

\node (s10) at (0,1.2) {};
\node(s12) at (0.6,1.2){};
\node(s14) at (1.2,1.2){};
\node(s16) at (1.8,1.2){};
\node(s18) at (2.4,1.2){};

\node (s20) at (0,0.9) {};
\node (s21) at (0.3,0.9) {};
\node(s22) at (0.6,0.9){};
\node (s23) at (0.9,0.9) {};
\node(s24) at (1.2,0.9){};
\node (s25) at (1.5,0.9) {};
\node(s26) at (1.8,0.9){};
\node (s27) at (2.1,0.9) {};
\node(s28) at (2.4,0.9){};
\node (s29) at (2.7,0.9) {};

\node (s30) at (0,0.6) {};
\node (s31) at (0.3,0.6) {};
\node(s32) at (0.6,0.6){};
\node (s33) at (0.9,0.6) {};
\node(s34) at (1.2,0.6){};
\node (s35) at (1.5,0.6) {};
\node(s36) at (1.8,0.6){};
\node (s37) at (2.1,0.6) {};
\node(s38) at (2.4,0.6){};
\node (s39) at (2.7,0.6) {};

\node (s40) at (0,0.3) {};
\node (s41) at (0.3,0.3) {};
\node(s42) at (0.6,0.3){};
\node (s43) at (0.9,0.3) {};
\node(s44) at (1.2,0.3){};
\node (s45) at (1.5,0.3) {};
\node(s46) at (1.8,0.3){};
\node (s47) at (2.1,0.3) {};
\node(s48) at (2.4,0.3) {};
\node (s49) at (2.7,0.3) {};

\node (s51) at (0.3,0.0) {};
\node (s53) at (0.9,0.0) {};
\node (s55) at (1.5,0.0) {};
\node (s57) at (2.1,0.0) {};
\node (s59) at (2.7,0.0) {};

\foreach \from/\to in {s10/s12, s12/s14,s14/s16,s16/s18}
\draw [-] (\from) -- (\to);

\foreach \from/\to in {s20/s21, s21/s22, s22/s23, s23/s24, s24/s25, s25/s26,s26/s27,s27/s28,s28/s29}
\draw [-] (\from) -- (\to);

\foreach \from/\to in {s30/s31, s31/s32, s32/s33, s33/s34, s34/s35, s35/s36,s36/s37,s37/s38,s38/s39}
\draw [-] (\from) -- (\to);

\foreach \from/\to in {s40/s41, s41/s42, s42/s43, s43/s44, s44/s45, s45/s46,s46/s47,s47/s48,s48/s49}
\draw [-] (\from) -- (\to);

\foreach \from/\to in {s51/s53, s53/s55,s55/s57,s57/s59}
\draw [-] (\from) -- (\to);

\foreach \from/\to in {s10/s20, s30/s40}
\draw [-] (\from) -- (\to);

\foreach \from/\to in {s21/s31,s41/s51}
\draw [-] (\from) -- (\to);

\foreach \from/\to in {s12/s22, s32/s42}
\draw [-] (\from) -- (\to);

\foreach \from/\to in {s23/s33,s43/s53}
\draw [-] (\from) -- (\to);

\foreach \from/\to in {s14/s24, s34/s44}
\draw [-] (\from) -- (\to);

\foreach \from/\to in {s25/s35,s45/s55}
\draw [-] (\from) -- (\to);

\foreach \from/\to in {s16/s26,s36/s46}
\draw [-] (\from) -- (\to);

\foreach \from/\to in {s27/s37,s47/s57}
\draw [-] (\from) -- (\to);

\foreach \from/\to in {s18/s28,s38/s48}
\draw [-] (\from) -- (\to);

\foreach \from/\to in {s29/s39,s49/s59}
\draw [-] (\from) -- (\to);

\end{tikzpicture}

\caption{$W_{5 \times 5}$}
\label{fig:5x5wall}
\end{figure}

Following the same line of thought, our motivation in this series is to study
induced subgraph obstructions  to bounded treewidth. In addition to subdivided walls mentioned above, complete graphs and complete bipartite graphs are easily observed to have arbitrarily large treewidth: the complete graph $K_{t+1}$ and the complete bipartite graph $K_{t,t}$ both have treewidth $t$. Line graphs of subdivided walls form another family of graphs with unbounded treewidth, where the {\em line graph} $L(F)$ of a graph $F$ is the graph with vertex set $E(F)$, such that two vertices of $L(F)$ are adjacent if the corresponding edges of $G$ share an end.  

We call a family $\mathcal{H}$ of graphs {\em useful} if
there exists an integer $c(\mathcal{H})$ such that every $\mathcal{H}$-free graph has treewidth at most $c(\mathcal{H})$. The discussion above can be summarized as follows:

\begin{theorem}
  \label{necessary} If $\mathcal{H}$ is a useful family of graphs,
  then there exists an integer $t$ such that   $\mathcal{H}$ contains
  $K_t,K_{t,t}$, an induced subgraph of each subdivision of $W_{t \times t}$ and
  an induced subgraph of the line graph
  of each subdivision of $W_{t \times t}$.
  \end{theorem}

The following was conjectured in \cite{aboulker} and proved in \cite{Korhonen}:
\begin{theorem}\label{boundeddeg} \cite{Korhonen}
For all $k, \Delta > 0$, there exists $c = c(k, \Delta)$ such that every graph with
maximum degree at most $\Delta$ and treewidth at least $c$ contains a subdivision of $W_{k\times k}$ or the
line graph of a subdivision of $W_{k \times k}$ as an induced subgraph.
\end{theorem}

The bounded-degree condition of Theorem \ref{boundeddeg} implies that $K_{\Delta+2}$ and $K_{\Delta+1, \Delta+1}$ are excluded. However, Theorem \ref{boundeddeg} does not hold if ``bounded degree'' is replaced by excluding $K_{\Delta+2}$ and $K_{\Delta+1, \Delta+1}$, as is evidenced by the constructions  of \cite {Davies} and  \cite{layered-wheels}. Thus a natural question arises:  what can replace this condition? 
Let us call a family $\mathcal{F}$ of graphs {\em helpful} if the following
holds: 
  for all $t > 0$, there exists $c = c(t)$ such that every
  $\mathcal{F}$-free graph
  with  treewidth more than
  $c$ contains $K_t$, $K_{t,t}$, a subdivision of $W_{t\times t}$ or the
line graph of a subdivision of $W_{t \times t}$.
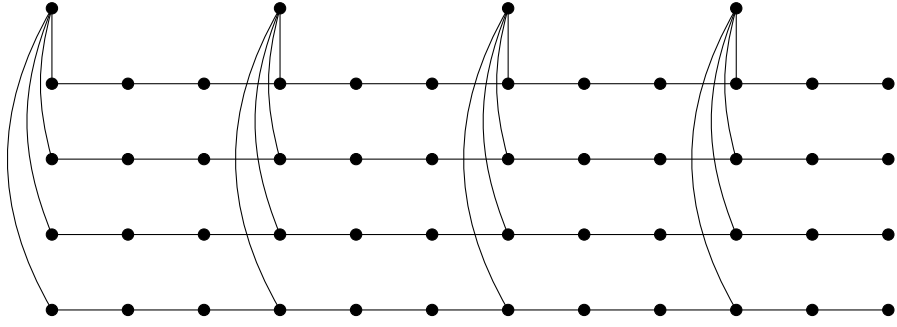
\begin{figure}[!htb]
\centering

\begin{tikzpicture}[scale=1,auto=left]
\tikzstyle{every node}=[inner sep=1.5pt, fill=black,circle,draw]  
\centering
        
\foreach \i in {1,...,4} {
        \draw (1, \i) -- (12, \i);
        \node() at (\i * 3 - 2, 5) {};
        
        \draw (\i * 3 - 2, 5) edge[out = -90, in = 90] (\i * 3 - 2, 4);
        
        \draw (\i * 3 - 2, 5) edge[out = -105, in = 105] (\i * 3 - 2, 3);
        
        \draw (\i * 3 - 2, 5) edge[out = -112, in = 112] (\i * 3 - 2, 2);
        
        \draw (\i * 3 - 2, 5) edge[out = -120, in = 120] (\i * 3 - 2, 1);

    }

\foreach \i in {1,...,12} {
            \foreach \x in {1,...,4} {
            \node() at (\i, \x) {};
                
            }
        }

\end{tikzpicture}

\caption{A wheel-free graph with large treewidth \cite{Davies}}
\label{fig:G34}
\end{figure}

A {\em hole} in a graph is an induced cycle of length at least four. The {\em length} of a hole is the number of vertices in it. A {\em wheel} is a graph consisting of a hole $C$ and a vertex $v$ with at least three neighbors in $C$ (in the literature, sometimes further restrictions are placed on the location of the neighbors of $v$ in $C$).
In view of the prevalence of wheels in the  construction of \cite{layered-wheels}, one might ask if the
family of all wheels is helpful. The answer to this question is negative, because of the construction of \cite{Davies} (see Figure \ref{fig:G34} for an example; we omit the precise definition), but the following weaker statement is true, and it is the main result of this paper. Let $\mathcal{T}_1$ be the family of
all graphs consisting of a hole $C$  and a vertex outside of $C$ with at least two neighbors in $C$. 
The class of $\mathcal{T}_1$-free graphs was studied in \cite{artv}; in  Section~\ref{sec:heavyseagulls} we strengthen their results.  A crucial difference between
Theorem~\ref{breakextheavyseagull} and \cite{artv} is that in \cite{artv} only the existence of
certain cutsets is shown, while we are able to guarantee that every heavy seagull is broken by a cutset of the required type (see Section~\ref{sec:heavyseagulls} for details).

Our main result in this paper is the following:
\begin{theorem}\label{helpful_main}
  The family $\mathcal{T}_1$ is helpful.
\end{theorem}

In fact, we prove something stronger. In the following, the \emph{length of a path} is its number of edges. 
A {\em pyramid} is a graph consisting of a vertex $a$, a triangle $\{b_1, b_2, b_3\}$, and three paths $P_i$ from $a$ to $b_i$ for $1 \leq i \leq 3$ of length at least one, such that for $i \neq j$ the only edge between
$P_i \setminus \{a\}$ and $P_j \setminus \{a\}$ is $b_ib_j$, and at most one of $P_1, P_2, P_3$ has length exactly one.

A {\em prism} is a graph consisting of two triangles $\{a_1,a_2,a_3\}$ and  $\{b_1, b_2, b_3\}$, and three paths $P_i$ from $a_i$ to $b_i$ for $1 \leq i \leq 3$,
all of length at least one, and such that for $i \neq j$ the only edges between
$P_i$ and $P_j$ are $a_ia_j$ and $b_ib_j$.

Let $\mathcal{T}_2$ be the family of
all graphs consisting of a hole  $C$  and a vertex outside of $C$ with at least two non-adjacent neighbors in $C$,
together with all prisms and all pyramids. Note that each graph in $\mathcal{T}_2$ contains a graph in $\mathcal{T}_1$ (so the class of $\mathcal{T}_1$-free graphs is properly contained in the class of $\mathcal{T}_2$-free graphs). We prove:

\begin{theorem}\label{helpful_main2}
  The family $\mathcal{T}_2$ is helpful.
\end{theorem}

Let us next restate Theorem \ref{helpful_main2} more explicitly.
A graph $G$ is {\em sparse} if 
for every hole  $H$ of $G$ and vertex
$v \not \in H$, there is an
 edge $ab$ of $H$ such that $N(v) \cap H \subseteq \{a,b\}$.
 A graph is {\em very sparse} if it is sparse and also
 (pyramid, prism)-free (thus a graph is very sparse if and only if it is $\mathcal{T}_2$-free).
 It follows that if $G$ is very sparse,
then $G$ does not contain $K_{3,3}$ or the line graph of a subdivision of
$W_{3 \times 3}$. Let $\mathcal{F}$ be the family of all very sparse graphs, and let
$\mathcal{F}_t$ be the family of all very sparse graphs with no clique of size at least $t+1$.

We prove:
\begin{theorem}
  \label{sparsethm}
  For all $t> 0$, there exists $c = c(t)$ such that every graph in
  $\mathcal{F}_t$ with 
  treewidth more than $c$ contains a subdivision of $W_{t\times t}$ (as an
  induced subgraph).
\end{theorem}

The rough outline of the proof of Theorem \ref{sparsethm} is as follows. Our first step is to  show that if a graph in $\mathcal{F}_t$ contains a triangle, then it admits a clique cutset. Thus it is enough to prove the result for graphs in $\mathcal{F}_2$.
 Now let $G \in \mathcal{F}_2$.
A {\em heavy seagull} in $G$  is an induced three-vertex path both of whose ends have degree at least three in $G$.
First we prove that every heavy seagull of $G$ is ``broken'' by a two-clique-separation (this means  that for every heavy seagull $H$ of $G$, there exist two cliques $K_1,K_2 \in G$ such that no component of
$G \setminus (K_1 \cup K_2)$ contains $H$).
Now the idea is to use the central bag method, developed in earlier papers in this series \cite{wallpaper,logpaper,pyramiddiamond,evenholetw}, to identify an induced subgraph $\beta $ of $G$ that contains no heavy seagull, and such that the treewidth of $G$ is not much larger than the treewidth of $\beta$. The key difference between our situation here and those in the earlier papers
is that the cutsets we use to break the heavy seagulls are not connected, a property that was crucial in the earlier proofs. To deal with this difficulty, we change the definition of a central bag, including in it a path
between the two cliques of the cutset  whose interior is in  $G \setminus \beta$ (this is in the spirit of, but different from,  ``marker paths'' for $2$-joins).  We then modify the previously known central bag tools to work in this new setting. By ``breaking'' heavy seagulls, we arrange that in $\beta$, vertices of degree at least three appear in components of bounded size. This in turn allows us to bound  the treewidth of $\beta$,  and theorem follows.

\subsection{Definitions and notation}
Let $G$ be a graph.
For $X \subseteq V(G)$, we denote by $G[X]$ the induced subgraph of $G$ with
vertex set $X$, and $G \setminus X$ denotes  $G[V(G) \setminus X]$.
In this paper we use the set  $X$ and the subgraph $G[X]$ of $G$
interchangeably. If $F$ is a graph and $G[X]$ is isomorphic to $F$, we say that
{\em $X$ is  an $F$ in $G$}.
Let $v \in V(G)$. The \emph{open neighborhood of $v$}, denoted $N(v)$, is the set of all vertices in $V(G)$ adjacent to $v$. We denote the degree of $v$ in $G$ by $\deg_G(v)=|N(v)|$. The \emph{closed neighborhood of $v$}, denoted $N[v]$, is $N(v) \cup \{v\}$. Let $X \subseteq V(G)$. The \emph{open neighborhood of $X$}, denoted $N(X)$, is the set of all vertices in $V(G) \setminus X$ with a neighbor in $X$. The \emph{closed neighborhood of $X$}, denoted $N[X]$, is $N(X) \cup X$. If $H$ is an induced subgraph of $G$ and
$X \subseteq V(G)$, then $N_H(X)= N(X) \cap H$.
Let $Y \subseteq V(G)$ be disjoint from $X$. Then, $X$ is {\em complete} to $Y$
if every vertex of $X$ is adjacent to every vertex of $Y$, and 
$X$ is \emph{anticomplete} to $Y$ if there are no edges between $X$ and $Y$. We use $X \cup v$ to mean $X \cup \{v\}$, and $X \setminus v$ to mean $X \setminus \{v\}$.
 
Given a graph $G$, a {\em path in $G$} is an induced subgraph of $G$ that is a path. If $P$ is a path in $G$, we write $P = p_1 \dd \hdots \dd p_k$ to mean that $p_i$ is adjacent to $p_j$ if and only if $|i-j| = 1$. We call the vertices $p_1$ and $p_k$ the \emph{ends of $P$}, and say that $P$ is \emph{from $p_1$ to $p_k$}. The \emph{interior of $P$}, denoted by $P^*$, is the set $P \setminus \{p_1, p_k\}$. The \emph{length} of a path $P$ is the number of edges in $P$.

A {\em theta} is a graph $T$ containing two vertices $a, b$ and three paths $P_1, P_2, P_3$ from $a$ to $b$ of length at least two, such that $P_1 \setminus \{a, b\}, P_2 \setminus \{a, b\}, P_3 \setminus \{a, b\}$ are  pairwise disjoint and anticomplete to each other. We call $a,b$ the {\em ends} of $T$.

\subsection{Organization of the paper}
This paper is organized as follows. In Section~\ref{sec:Separations}, we give general background and definitions related to separations in graphs; we also discuss connections between different kinds of separations in the special case of sparse graphs. In Section~\ref{sec:reduction}, we reduce Theorem \ref{sparsethm} to the case of triangle-free sparse graphs.  In Section~\ref{sec:centralbags}, we discuss balanced separators in graphs, and develop our main tool, Theorem \ref{unbalancedbeta}, which allows us to use the
central bag method. In Section~\ref{sec:2clique}, we prove results about two-clique-separations, which are the cutsets that will be used to form the central bag. In Section~\ref{sec:heavyseagulls}, we prove structural results that allow us to break every heavy seagull in a triangle-free sparse graph and produce a central bag that contains no heavy seagulls.
In Section~\ref{sec:trianglefree}, we use the tools of Section~\ref{sec:centralbags} to prove our main result for graphs in $\mathcal{F}_2$. Finally, in Section~\ref{sec:theend}, we prove Theorem \ref{sparsethm}.

\section{Separations}
\label{sec:Separations} 
A {\em separation} of a graph $G$ is a triple $(A, C, B)$, where $A, B, C \subseteq V(G)$, $A \cup C \cup B = V(G)$, $A$, $B$, and $C$ are pairwise disjoint, and $A$ is anticomplete to $B$. If $S = (A, C, B)$ is a separation, we let $A(S) = A$, $B(S) = B$, and $C(S) = C$. 
We say that $C \subseteq V(G)$ is a {\em cutset} of $G$ if there
exists a separation $(A,C,B)$ of $G$ with $A \neq \emptyset$ and $B \neq \emptyset$.
A {\em clique} in a graph is a (possibly empty) set of pairwise adjacent vertices.  We say that $G$ admits a {\em clique cutset} if there is a cutset of $G$
that is a clique (in particular every disconnected graph admits a clique cutset).
A separation $(A, C, B)$ is a {\em star separation} if there exists $v \in C$ such that $C \subseteq N[v]$ (we say that $v$ is a {\em center} of $C$). A star separation $(A,C,B)$ is {\em proper} if
$A \neq \emptyset$ and $B \neq \emptyset$.
We say that $G$ admits a {\em star cutset}
if there is a proper star separation in $G$.

First we observe:

\begin{lemma}
  \label{startoclique}
  Let $G$ be a sparse graph and  $(A,C,B)$ be a separation of $G$
  with $A \neq \emptyset$ and $B \neq \emptyset$. Suppose that there
  exist $v_1, \dots, v_k \in C$ such that $C \subseteq \bigcup_{i=1}^k N[v_i]$.
  Let $D_1$ be a component of $A$ and let $D_2$ be a component of $B$.
  Then there exist cliques $X_1, \dots, X_k \subseteq C$ of $G$ such that every path
  from a
  vertex of $D_1$ to a vertex of $D_2$ meets $\bigcup_{i=1}^k X_i$. In particular,
  if $G$ admits a star cutset, then $G$ admits a clique cutset.
\end{lemma}

\begin{proof}
  Let $N_1=N(D_1) \subseteq C$, and let $D_2'$ be the component of $G \setminus (N_1 \cup \{v_1, \dots, v_k\})$ such that $D_2 \subseteq D_2'$. Let $X=N(D_2') \cup \{v_1, \dots, v_k\}$. Then
  $X \subseteq N_1 \cup \{v_1, \dots, v_k\} \subseteq C$, and every path from a vertex of $D_1$ to a vertex of $D_2'$ in $G$ meets $X$. We claim that for every
  $i \in \{1, \dots, k\}$ the set $X \cap N[v_i]$ is a clique.
    Suppose not, and 
  let $x,y \in X \cap N[v_1]$ (say) be non-adjacent (and so in particular, $x, y \neq v_1$). It follows that $x, y \in N(D_1) \cap N(D_2')$. Let $P_1$ be a path from $x$ to $y$
  with $P_1^* \subseteq D_1$ and let  $P_2$ be a path from $x$ to $y$
  with $P_2^* \subseteq D_2'$.
  Then $H=x \dd P_1 \dd y \dd P_2 \dd x$ is a hole and $v_1 \not \in H$ since $v_1 \in X$.
  But now $v_1$ has two non-adjacent neighbors in $H$, contrary to the fact that $G$ is sparse.
  This proves Lemma \ref{startoclique}.
  \end{proof}

 Lemma 7 from \cite{cliquetw} shows that clique cutsets do not affect treewidth. Now, by Lemma \ref{startoclique}, it follows that in order to prove Theorem \ref{sparsethm} it is enough to prove the following:

\begin{theorem}
  \label{sparsethmnostar}
  For all $t> 0$, there exists $c = c(t)$ such that every
graph in $\mathcal{F}_t$ with 
   treewidth more than $c$ and with no star cutset
contains a subdivision of $W_{t\times t}$ as an induced subgraph.
\end{theorem}

\section{Reducing to the triangle-free case}
\label{sec:reduction}
In this section we show how to deduce  Theorem \ref{sparsethm} from the special case
of triangle-free graphs. A {\em diamond} is the graph obtained from $K_4$ by removing an
edge.

\begin{lemma}
  \label{diamondlemma}
  Let $G$ be a sparse graph  and assume that $G$ does not admit a star cutset.
  Then $G$ is diamond-free. 
\end{lemma}

\begin{proof}
  Suppose first $\{a,b,c,d\}$ is a diamond in $G$. We may assume that
  the pair $ac$ is non-adjacent.
 Since $b$ is not the center of a star cutset in $G$, it follows that 
  there exists is a path from $a$ to $c$ with no neighbor of
  $b$ in its interior. Let $P$ be such a path. Then $d$ is not a vertex of $P$, since $d$ is adjacent to $b$. Moreover, $a \dd P \dd c \dd b \dd a$ is a hole, and $d$ has three neighbors in it, namely $a$, $b$  and $c$,
  a contradiction. This proves that $G$ is diamond-free. 
\end{proof}

We also need the following folklore result that appeared in \cite{prismfree}:

\begin{lemma}\label{minimalconnected}
Let $x_1, x_2, x_3$ be three distinct vertices of a graph $G$. Assume that $H$ is a connected induced subgraph of $G \setminus \{x_1, x_2, x_3\}$ such that $V(H)$ contains at least one neighbor of each of $x_1$, $x_2$, $x_3$, and that $V(H)$ is minimal subject to inclusion. Then, one of the following holds:
\begin{enumerate}[(i)]
\item For some distinct $i,j,k \in  \{1,2,3\}$, there exists $P$ that is either a path from $x_i$ to $x_j$ or a hole containing the edge $x_ix_j$ such that
\begin{itemize}
\item $V(H) = V(P) \setminus \{x_i,x_j\}$, and
\item either $x_k$ has two non-adjacent neighbors in $H$ or $x_k$ has exactly two neighbors in $H$ and its neighbors in $H$ are adjacent.
\end{itemize}

\item There exists a vertex $a \in V(H)$ and three paths $P_1, P_2, P_3$, where $P_i$ is from $a$ to $x_i$, such that 
\begin{itemize}
\item $V(H) = (V(P_1) \cup V(P_2) \cup V(P_3)) \setminus \{x_1, x_2, x_3\}$, and 
\item the sets $V(P_1) \setminus \{a\}$, $V(P_2) \setminus \{a\}$ and $V(P_3) \setminus \{a\}$ are pairwise disjoint, and
\item for distinct $i,j \in \{1,2,3\}$, there are no edges between $V(P_i) \setminus \{a\}$ and $V(P_j) \setminus \{a\}$, except possibly $x_ix_j$.
\end{itemize}

\item There exists a triangle $a_1a_2a_3$ in $H$ and three paths $P_1, P_2, P_3$, where $P_i$ is from $a_i$ to $x_i$, such that
\begin{itemize}
\item $V(H) = (V(P_1) \cup V(P_2) \cup V(P_3)) \setminus \{x_1, x_2, x_3\} $, and 
\item the sets $V(P_1)$, $V(P_2)$ and $V(P_3)$ are pairwise disjoint, and
\item for distinct $i,j \in \{1,2,3\}$, there are no edges between $V(P_i)$ and $V(P_j)$, except $a_ia_j$ and possibly $x_ix_j$.
\end{itemize}
\end{enumerate}
\end{lemma}

\begin{lemma}
  \label{cliqueortrianglefree}
  Let $G \in \mathcal{F}$. Then either $G \in \mathcal{F}_2$, $G$ is a complete graph, or
  $G$ admits a star cutset.
\end{lemma}

\begin{proof}
We may assume that $G$ does not admit a star cutset and $G$ is not a complete graph.
  Let $K$ be an inclusion-wise maximal clique of $G$ with $|K|>2$, and let
  $D=G \setminus K$. Since $G$ does not admit a clique cutset and is not a complete graph, it follows that
  $D$ is connected, non-empty, and every vertex of $K$ has a neighbor in $D$. By Lemma \ref{diamondlemma}, it follows that $G$ does not contain a diamond.

\sta{\label{onenbrinK} Let $v \in D$. Then $v$ has at most one neighbor in $K$.}

 Assume that $v$ has at least two neighbors in $K$, say $k_1$ and $k_2$.
  Since $K$ is a maximal clique, there exists $k_3 \in K$ non-adjacent to $v$.
But now $\{v,k_1,k_2,k_3\}$ is a diamond, a contradiction.
    This proves \eqref{onenbrinK}.
  \\
  \\
  Now let $x_1,x_2,x_3 \in K$. Apply Lemma \ref{minimalconnected} to
  $\{x_1,x_2,x_3\}$ and a minimal connected subgraph $H$ of $D$ containing at least  one neighbor of each of $x_1,x_2,x_3$. By \eqref{onenbrinK}, we have that $|V(H)| \geq 3$. Now
  the first outcome of Lemma \ref{minimalconnected} gives a hole and a vertex with
  two non-adjacent neighbors in it, the second outcome gives a pyramid, and
  the third  gives a prism.  In all cases we get a
  contradiction to the fact that $G \in \mathcal{F}$. This proves Lemma
  \ref{cliqueortrianglefree}.
\end{proof}

Now, by Lemma
  \ref{cliqueortrianglefree}, in order to prove Theorem \ref{sparsethmnostar} it is enough to prove:
\begin{theorem}
  \label{t=2}
  For all $k$, there exists $c = c(k)$ such that every
graph in $\mathcal{F}_2$ with no star cutset and with    treewidth more than $c$ 
contains a subdivision of $W_{k\times k}$ as an induced subgraph.
\end{theorem}

\section{Balanced separators and central bags}
\label{sec:centralbags}
Let $G$ be a graph, and let $w: V(G) \to [0, 1]$. For $X \subseteq V(G)$, we write $w(X)$ for $\sum_{x \in X} w(x)$. We call $w$ a \emph{weight function} on $G$ if $w(G) = 1$. Now let $c \in [\frac{1}{2}, 1)$. A set $X \subseteq V(G)$ is a  {\em $(w, c)$-balanced separator} if $w(D) \leq c$ for every component $D$ of $G \setminus X$. The next two lemmas show how $(w, c)$-balanced separators relate to treewidth. The first result was originally proved by Harvey and Wood in \cite{params-tied-to-tw} using different language, and was restated and proved in the language of $(w, c)$-balanced separators in \cite{wallpaper}. 
\begin{lemma}[\cite{wallpaper,params-tied-to-tw}]\label{lemma:bs-to-tw}
Let $G$ be a graph, let $c \in [\frac{1}{2}, 1)$, and let $k$ be a positive integer. If $G$ has a $(w, c)$-balanced separator of size at most $k$ for every weight function $w$ on $G$, then $\tw(G) \leq \frac{1}{1-c}k$. 
\end{lemma}

\begin{lemma}[\cite{cygan}]
\label{lemma:tw-to-weighted-separator}
Let $G$ be a graph and let $k$ be a positive integer. If $\tw(G) \leq k$, then $G$ has a $(w, c)$-balanced separator of size at most $k+1$ for every $c \in [\frac{1}{2}, 1)$ and for every weight function $w$ on $G$.
\end{lemma}

A pair $(G,w)$ is {\em $d$-unbalanced} if $w$ is a weight function on $G$, and
$G$ has no $(w,\frac{1}{2})$-balanced separator of size at most $d$
(if there is a $(w,\frac{1}{2})$-balanced separator of size at most $d$, we say that
$(G,w)$ is {\em $d$-balanced}).

Let $K$ be an integer, let $G$ be a graph and let $K_1,K_2$ be two cliques of
$G$, each of size at most $K$. Let $(G,w)$ be a $2K$-unbalanced pair.
Following \cite{evenholetw},  we define the {\em canonical two-clique-separation}
for $\{K_1,K_2\}$, as follows.
Let $B(K_1,K_2)$ be a component of $G \setminus (K_1 \cup K_2)$ with
$w(B(K_1,K_2))$ maximum.
Since $(G,w)$ is  $2K$-unbalanced, it follows that $K_1 \cup K_2$ is not a $(w,\frac{1}{2})$-balanced separator; consequently   $w(B(K_1,K_2)) > \frac{1}{2}$, and so the choice of
$B(K_1,K_2)$
is unique.
Let $A(K_1,K_2)=G \setminus (B(K_1,K_2) \cup K_1 \cup K_2)$ and $C(K_1,K_2)=K_1 \cup K_2$.  Now
$S(K_1,K_2)=(A(K_1,K_2),C(K_1,K_2),B(K_1,K_2))$ is the
canonical two-clique-separation corresponding to $\{K_1,K_2\}$.

For the remainder of this section, let  $K$ be an integer, and let $(G,w)$ be a  $2K$-unbalanced pair. 
Let $K_1^1, K_2^1, K_1^2, K_2^2$ be cliques in $G$.
For $i \in \{1,2\}$, let $S_i=(A_i, C_i, B_i)$ be the canonical two-clique-separation for  $\{K_1^i,K_2^i\}$.
We say that $(A_1, C_1, B_1)$ and $(A_2, C_2, B_2)$ are {\em non-crossing} if $A_1 \cup C_1 \subseteq B_2 \cup C_2$ and
$A_2 \cup C_2 \subseteq B_1 \cup C_1$, and that
$(A_1, C_1, B_1)$ and $(A_2, C_2, B_2)$ are {\em loosely non-crossing} if
$A_1 \cap C_2= A_2 \cap C_1 = \emptyset$. Clearly, if
$S_1$ and $S_2$  are non-crossing, then they are loosely non-crossing.
(Note that here we break the symmetry between $A_i$ and $B_i$, and so our definition is slightly different from the classical definition of \cite{RS-GMV}.)

The following observation follows immediately from the definition of a canonical two-clique-separation.

\begin{lemma}
  \label{compattachments}
  Assume that $G$ does not admit a
  star cutset. Let $K_1,K_2$ be cliques of size at most $K$  in $G$ such that
  $A(K_1,K_2) \neq \emptyset$.
  Then the  following hold.
  \begin{enumerate}
    \item $K_1 \cap K_2 = \emptyset$. 
    \item Let $D$ be a component of
      $G \setminus (K_1 \cup K_2)$. Then  $N(D) \cap K_i \neq \emptyset$ for all
      $i \in \{1,2\}$, and so
      there is a path from a vertex of $K_1$ to a vertex of $K_2$
      with non-empty interior in $D$.
      \end{enumerate}
\end{lemma}


 Throughout this section, let $\mathcal S$ be a set of sets $\{K_1, K_2\}$ where each of $K_1, K_2$ is a clique of size at most $K$ of $G$, and let $\mathcal T$ be the set of canonical two-clique-separations corresponding to members of $\mathcal S$. Moreover, we will assume each pair of separations in $\mathcal T$ is loosely non-crossing.
	   
	   We would now like to define a {\em central bag} for $\mathcal S$. Roughly speaking, this central bag is the intersection of the heavy blocks $B(S) \cup C(S)$ of the separations, together with some paths that capture the important $w$-related information about the light blocks. In order to define it, we start by considering the connected components of the union $\bigcup_{S \in \mathcal T} A(S)$ of the light sides of the separations. We first note that, given such a component $D$ and an $S_0 \in \mathcal T$, we either have $D \subseteq A(S_0)$ or $D \cap A(S_0) = \emptyset$. Indeed, $N(A(S_0)) \subseteq C(S_0)$, and so if $D$ simultaneously contains vertices in $A(S_0)$ and vertices not in $A(S_0)$, then $D \setminus A(S_0)$ must contain vertices in $C(S_0)$; but $D \setminus A(S_0) \subseteq \bigcup_{S \in \mathcal T: S \neq S_0} A(S)$, which has empty intersection with $C(S_0)$ by the loosely non-crossing property -- a contradiction. 
	    
	   We now want to ``reorganize'' the $A(S)$ by assigning each component of $\bigcup_{S \in \mathcal T} A(S)$ to a unique $A(K_1, K_2)$ in a consistent way. To that end, we fix a total order $\pi$ on $\mathcal S$, and group the components according to the $\pi$-minimal $\{K_1, K_2\}$ to whose $A(S)$ they belong. Specifically, for $\{K_1, K_2\} \in \mathcal S$, we let $A^*(K_1, K_2)$ be the union of all components $D$ of $A(K_1, K_2)$ such that for all $\{K_1', K_2'\} \in \mathcal S$ with $D \subseteq A(K_1', K_2')$, $\pi(A(K_1, K_2)) \leq \pi(A(K_1', K_2'))$. 
	   
	   Now, by Lemma~\ref{compattachments}, for every $\{K_1, K_2\}$ with $A^*(K_1, K_2) \neq \emptyset$, there exists a path $P_{K_1K_2}^*$ in $A^*(K_1, K_2)$ whose two (possibly coinciding) endpoints have a neighbour in $K_1$ and in $K_2$ respectively. Let $\mathcal{S}'=\{\{K_1,K_2\} \in \mathcal{S} \mid A^*(K_1,K_2) \neq \emptyset \}$, and write $$\beta=
  \bigcap_{\{K_1,K_2\} \in \mathcal{S}'} (B(K_1,K_2) \cup K_1 \cup K_2)
 \cup
  \bigcup_{\{K_1,K_2\} \in \mathcal{S}'}P_{K_1K_2}^*.$$

We call $\beta$ a {\em central bag} for $\mathcal{S}$.
Note that the choice of $\beta$ is not unique since the choice of
the paths $P_{K_1K_2}^*$ is not unique.
Observe that  $\bigcap_{\{K_1,K_2\} \in \mathcal{S}'} (B(K_1,K_2) \cup K_1 \cup K_2)=V(G) \setminus 
\bigcup_{S  \in \mathcal{T}}A(S)$.

Let $w_\beta$ be the function on $\beta$ defined as follows.
For $v \in \bigcap_{\{K_1,K_2\} \in \mathcal{S}'}(B(K_1,K_2) \cup K_1 \cup K_2)$, we
 set $w_\beta(v)=w(v)$.
Next let $\{K_1,K_2\} \in \mathcal{S}'$,
and let $a_{K_1,K_2}$ be the endpoint of $P_{K_1K_2}^*$ adjacent to a vertex of $K_1$; set
$w_\beta(a_{K_1,K_2})=w(A^*(K_1,K_2))$.
Let $w_\beta(v)=0$ for every $v \in \beta$ where $w_\beta$ has not been defined
yet. We call $w_\beta$ the {\em weight function inherited from $w$}.

\begin{lemma}
\label{wbetanormal}
The function $w_\beta$ is a weight function, that is, $w_\beta(\beta)=1$.
\end{lemma}

\begin{proof}

We note that, for any $\mathcal S_0 \subseteq \mathcal S$, the pair of sets $\bigcap_{\{K_1, K_2\} \in \mathcal S_0}(B(K_1, K_2) \cup C(K_1, K_2))$ and $\bigcup_{\{K_1, K_2\} \in \mathcal S_0} A(K_1, K_2)$ partition $V(G)$. In particular, $$w(G) = w\left(\bigcap_{\{K_1, K_2\} \in \mathcal S'} B(K_1, K_2) \cup C(K_1, K_2)\right) + w\left(\bigcup_{\{K_1, K_2\} \in \mathcal S'} A(K_1, K_2) \right).$$

	    Moreover, by construction, $(A^*(K_1, K_2))_{\{K_1, K_2\} \in \mathcal S'}$ is a partition of $\bigcup_{\{K_1, K_2\} \in \mathcal S'} A(K_1, K_2)$, so that $$w(G) = w\left(\bigcap_{\{K_1, K_2\} \in \mathcal S'} B(K_1, K_2) \cup C(K_1, K_2)\right) + \sum\limits_{\{K_1, K_2\} \in \mathcal S'} w(A^*(K_1, K_2)).$$
	    
	    Since each $A^*(K_1, K_2)$ with $\{K_1, K_2\} \in \mathcal S'$ contains exactly one of the vertices $a_{K_1, K_2}$, we have $$\sum\limits_{\{K_1, K_2\} \in \mathcal S'} w(A^*(K_1, K_2)) = \sum\limits_{\{K_1, K_2\} \in \mathcal S'} w_\beta(a_{K_1, K_2}).$$
	    
	    Putting everything together, we obtain:
	    \begin{align*}
	        w_\beta(\beta) &= w_\beta\left(\bigcap_{\{K_1, K_2\} \in \mathcal S'} B(K_1, K_2) \cup C(K_1, K_2)\right) + \sum\limits_{\{K_1, K_2\} \in \mathcal S'} w_\beta(a_{K_1, K_2}) \\
	        &= w\left(\bigcap_{\{K_1, K_2\} \in \mathcal S'} B(K_1, K_2) \cup C(K_1, K_2)\right) + \sum\limits_{\{K_1, K_2\} \in \mathcal S'} w(A^*(K_1, K_2)) \\
	        &= w(G) = 1.
	    \end{align*}

\end{proof} 

For $v \in V(G)$, let
$$\delta_{\mathcal{S}}(v)=\bigcup_{K \text {: }v \in K \text{ and there exists } L \text{ such that }\{K,L\} \in \mathcal{S}}K.$$

\begin{theorem}
\label{unbalancedbeta}
Let $d, \Delta$ be integers. Assume that
$|\delta_{\mathcal{S}}(v)| \leq \Delta$ for every $v \in G$.
Assume also that $(\beta, w_\beta)$ is $d$-balanced.
Then $(G,w)$ is $\max(2Kd,\Delta d)$-balanced.
\end{theorem}

\begin{proof}
Suppose that $X$ is a $(w_\beta,\frac{1}{2})$-balanced separator in $\beta$
with $|X| \leq d$.  We now construct a $(w, \frac{1}{2})$-balanced separator $Y$ of $G$
with $|Y| \leq \max(2Kd,\Delta d)$.

Let 
$$Y_1=X \cap   \left(\bigcap_{\{K_1,K_2\} \in \mathcal{S}'}(B(K_1,K_2) \cup K_1 \cup K_2)\right).$$
For $x \in Y_1$, let $Y(x)=\delta_{\mathcal{S}}(x)$.
Now let $x \in X \setminus Y_1$. It follows from the definition of $A^*(K_1,K_2)$ and $P_{K_1K_2}^*$ that $x \in P_{K_1K_2}^*$ for exactly one  $\{K_1,K_2\} \in \mathcal{S}'$;
let $Y(x)=K_1 \cup K_2$. Let $Y= \bigcup_{x \in X}Y(x)$. Then 
$|Y| \leq  \Delta |Y_1|  +2K(d -|Y_1|) \leq \max(\Delta d,2Kd)$, as required. Next we prove that $Y$ is a $(w, \frac{1}{2})$-balanced separator of $G$. 

\sta{\label{comps-of-G-beta} Let $F$ be a component of $G \setminus \beta$. Then, there exists $\{K_1,K_2\} \in \mathcal{S}$ such that $F \subseteq A^*(K_1,K_2)$.}

By construction of $\beta$, it holds that $G \setminus \beta \subseteq \bigcup_{\{K_1,K_2\} \in \mathcal{S}} A(K_1,K_2)$; consequently there exists $\{K_1,K_2\} \in \mathcal{S}$ such that $F \subseteq A^*(K_1,K_2)$. This proves \eqref{comps-of-G-beta}. \\

From now on, let $D$ be a component of $G \setminus Y$. We will show that $w(D) \leq \frac{1}{2}$. Since $(G, w)$ is $2K$-unbalanced, it follows that $w(A(K_1,K_2)) < \frac{1}{2}$ for all $\{K_1, K_2\} \in \mathcal{S}$, and so if $D$ is a component of $G \setminus \beta$, then by \eqref{comps-of-G-beta}, it follows that $w(D) \leq \frac{1}{2}$. Thus we may assume that $D \cap \beta \neq \emptyset$. 

Suppose first that $D \cap A(K_1,K_2) \neq \emptyset$ for some $\{K_1,K_2\} \in \mathcal{S}$ such that $K_1 \cup K_2 \subseteq Y$. Since $N(A(K_1,K_2)) \subseteq K_1 \cup K_2$ and $K_1 \cup K_2 \subseteq Y$, it follows that $D \subseteq A(K_1,K_2)$, and so $w(D) < \frac{1}{2}$. Therefore, we may assume that $D \cap A(K_1,K_2) = \emptyset$ for all $\{K_1, K_2\} \in \mathcal{S}$ such that $K_1 \cup K_2 \subseteq Y$. Next, suppose $D \cap A(K_1,K_2) \neq \emptyset$ for $\{K_1,K_2\} \in \mathcal{S}'$ such that
$P_{K_1K_2}^* \cap X \neq \emptyset$.
Let $x \in P_{K_1K_2}^* \cap X$. Now, $x \in X \setminus Y_1$, and so $Y(x) = K_1 \cup K_2 \subseteq Y$, a contradiction. Therefore, we may assume that for all $\{K_1,K_2\} \in \mathcal{S}'$ such that $D \cap A(K_1,K_2) \neq \emptyset$, it holds that
$P_{K_1K_2}^*$ is disjoint from $X$, and thus $P_{K_1K_2}^*$ is contained in a component of $\beta \setminus X$. Let $Q_1, \hdots, Q_m$ be the components of $\beta \setminus X$. 

\sta{\label{eq:markerpath} Let $\{K_1, K_2\} \in \mathcal{S}'$, and suppose that $P^*_{K_1 K_2} \subseteq Q_k$. Then $K_1 \cup K_2 \subseteq Q_k \cup Y$.}

Since $N(P^*_{K_1K_2}) \cap K_i \neq \emptyset$ for each $i \in \{1, 2\}$, it follows that each of $K_1,K_2$ either is contained in $Q_k$ or has a vertex in $X$.
Since 
every two separations in $\mathcal{T}$ are loosely non-crossing, it follows
that each of $K_1,K_2$ is either contained in $Q_k$ or has a vertex in $Y_1$.
Since $\delta_{\mathcal{S}}(x) \subseteq Y$ for every $x \in Y_1$, it follows that for $i \in \{1, 2\}$,
if $K_i \cap Y_1 \neq \emptyset$, then $K_i \subseteq Y$.
This proves \eqref{eq:markerpath}.

\sta{\label{eq:component} Let $\{K_1, K_2\} \in \mathcal{S}'$, and suppose that $N(A(K_1, K_2)) \cap Q_k \neq \emptyset$. Then either $K_1 \cup K_2 \subseteq Y$, or $P_{K_1K_2}^* \subseteq Q_k$. In particular, if $K_1 \cup K_2 \not\subseteq Y$, then there is at most one $k \in \{1, \dots, m\}$ with $N(A(K_1, K_2)) \cap Q_k \neq \emptyset$.}

If $P_{K_1K_2}^* \cap X \neq \emptyset$, then $K_1 \cup K_2 \subseteq Y$, and \eqref{eq:component} holds; so we may assume that $P_{K_1K_2}^* \cap X = \emptyset$, and since $P_{K_1K_2}^*$ is connected, it follows that $P_{K_1K_2}^* \subseteq Q_{k'}$ for some $k' \in \{1, \dots, m\}$. If $k = k'$, then \eqref{eq:component} holds, so we may assume that $k \neq k'$. It follows from \eqref{eq:markerpath} that $K_1 \cup K_2 \subseteq Q_{k'} \cup Y$ and that $K_1 \cup K_2 \subseteq Q_{k'} \neq \emptyset$, and thus $N(A(K_1, K_2)) \subseteq Q_{k'} \cup Y$, a contradiction. This proves \eqref{eq:component}. \\

Since $D \cap \beta \neq \emptyset$, it follows that for each $\{K_1, K_2\} \in \mathcal{S}'$ with $D \cap A(K_1, K_2) \neq \emptyset$, we have $D \cap N(A(K_1, K_2)) \neq \emptyset$, and in particular $(K_1 \cup K_2) \cap D \neq \emptyset$, so $K_1 \cup K_2 \not\subseteq Y$. Moreover, from \eqref{eq:component}, it follows that $P_{K_1K_2}^* \subseteq Q_k$ for some $k \in \{1, \dots, m\}$, and $N(A(K_1, K_2)) \cap Q_{k'} = \emptyset$ for all $k' \neq k$. Since $D$ is connected, it follows that there is a $k \in \{1, \dots, m\}$ such that for every $\{K_1, K_2\} \in \mathcal{S}'$ with $D \cap A(K_1, K_2) \neq \emptyset$, we have $N(A(K_1, K_2)) \subseteq Q_k \cup Y$, and $P^*_{K_1K_2} \subseteq Q_k$. It follows that $D \cap \beta \subseteq Q_k$, and $a_{K_1, K_2} \in Q_k$ for all such $\{K_1, K_2\} \in \mathcal{S}'$, and therefore $w(D) \leq w_\beta(Q_k) \leq \frac{1}{2}$. This concludes the proof. 
\end{proof}

Let $K_1,K_2$ be cliques of size at most $K$ in $G$.
We say that $S(K_1,K_2)$ is {\em proper} (or that the pair
$\{K_1,K_2\}$ is {\em proper}) if 
\begin{itemize}
    \item  some component $D$
of $A(K_1,K_2)$ satisfies $K_1 \cup K_2 \subseteq N(D)$, and 
\item if $|K_1|=|K_2|=1$, then   $A(K_1,K_2) \cup K_1 \cup K_2$ is not a path
from the vertex of $K_1$ to the vertex of $K_2$. 
\end{itemize}
We observe:

 \begin{lemma}
   \label{properdeg3}
   Let $K_1,K_2$ be cliques of size at most $K$ in $G$ and assume that
   $S(K_1,K_2)$ is a proper canonical two-clique-separation in $G$. Then either some
   vertex of $A(K_1,K_2)$ has at least three neighbors in
   $A(K_1,K_2) \cup K_1 \cup K_2$, or some vertex of $K_1 \cup K_2$ has
   at least two neighbors in $A(K_1,K_2)$.

 \end{lemma}

 \begin{proof}
   Let $D$ be a component of $A(K_1,K_2)$ such that
   $K_1 \cup K_2 \subseteq N(D)$. Then $N[D]$ has a spanning tree $T$  such that
   every vertex of $K_1 \cup K_2$ is a leaf of $T$.  If $|K_1|>1$, then
   $T$ has at least three leaves, and therefore some vertex of $D$ has
   degree at least three in $N[D]$ as required. Thus we may assume that
   $|K_1|=|K_2|=1$. If $N[D]$ is not a path from the vertex of $K_1$ to the
   vertex of $K_2$, then some vertex of $D$ has at least three neighbors in
   $N[D]$, and again theorem holds. Thus we may assume that
   $N[D]$ is a path from the vertex of $K_1$ to the
   vertex of $K_2$. Since $S(K_1,K_2)$ is proper, $A(K_1, K_2) \neq D$.
   Let $D'$ be a component of $A(K_1,K_2) \setminus D$. By Lemma \ref{compattachments}, we have that
   $K_1 \subseteq N(D')$. But then the vertex of $K_1$ has at least two neighbors in
   $A(K_1,K_2)$ as required. 
\end{proof}

 We say that $S(K_1,K_2)$ is {\em active} (or that the pair
 $\{K_1,K_2\}$ is {\em active}) if it is proper and for every 
 pair of cliques $K_1',K_2'$ of size at most $K$ in $G$ 
 such that $S(K_1',K_2')$ is proper and $K_1 \cup K_2 \neq K_1' \cup K_2'$, it holds that
 \begin{itemize}
 \item  $B(K_1',K_2') \cup K_1' \cup K_2'$ is not a proper subset of $B(K_1,K_2) \cup K_1 \cup K_2$; and
   \item if $B(K_1',K_2') \cup K_1' \cup K_2' = B(K_1,K_2) \cup K_1 \cup K_2$,
 then $B(K_1',K_2') \subset B(K_1,K_2)$.
\end{itemize}

  \begin{lemma}
    \label{active}
    Let $K_1,K_2$ be cliques of $G$ of size at most $K$.
    If $S(K_1,K_2)$ is active, then $K_1 \cup K_2 \subseteq N(B(K_1,K_2))$.
      \end{lemma}

  \begin{proof}
    Suppose not. We may assume that  there exists $x \in K_1$ such $x$  has no
    neighbor in $B(K_1,K_2)$. Then $(A(K_1,K_2) \cup \{x\}, (K_1 \cup K_2) \setminus \{x\}, B(K_1,K_2))$ is a  proper two-clique-separation of $G$ contrary to  the fact that
    $S$ is active.
    \end{proof}

\section{Two-clique-separations}
\label{sec:2clique}

The main result of this section will allow us to apply Theorem
  \ref{unbalancedbeta} with $K=2$:
 
  \begin{theorem}
  \label{noncrossing2cliques}
  Let $G \in \mathcal{F}_2$ and let $(G,w)$ be an $8$-unbalanced pair. Let $K_1,K_2,K_1',K_2'$  be cliques of $G$  such that  the separations $S=S(K_1,K_2)$ and $S'=S(K_1',K_2')$ are active
  in $G$. 
  Assume also that $G$ admits  no star cutset.
  Then $S$ and $S'$ are loosely non-crossing.
\end{theorem}

  \begin{proof}
  Suppose  that $S$ and $S'$ are not loosely non-crossing.
  Then 
  $(C(K_1,K_2) \cup C(K_1',K_2')) \cap (A(K_1,K_2) \cup A(K_1',K_2')) \neq \emptyset$. Since $w(B(K_1, K_2)) > \frac{1}{2}$ and $w(B(K_1', K_2')) > \frac{1}{2}$, it follows that $B(K_1, K_2) \cap B(K_1', K_2') \neq \emptyset$. 
  
  \sta{\label{CB'}$C(K_1,K_2) \cap B(K_1',K_2') \neq \emptyset$.}

  Suppose $C(K_1,K_2) \cap B(K_1',K_2') =\emptyset$. Since $B(K_1',K_2')$
  is connected, it follows that  $A(K_1,K_2) \cap B(K_1',K_2')  =  \emptyset$.
  Since by Lemma \ref{active} every vertex of $K_1' \cup K_2'$ has a neighbor in $B(K_1',K_2')$
  it follows that  $A(K_1,K_2) \cap C(K_1',K_2')  =  \emptyset$.
  But now $
  B(K_1',K_2') \cup K_1' \cup K_2' \subseteq B(K_1,K_2) \cup K_1 \cup K_2$. Since $S$ is active, it follows that $B(K_1',K_2') \cup K_1' \cup K_2' = B(K_1,K_2) \cup K_1 \cup K_2$. But now one of $S$, $S'$ is not active by the second bullet of the definition of being active, a contradiction. This proves \eqref{CB'}.
  \\
  \\
\sta{\label{CA'} $C(K_1',K_2') \cap A(K_1,K_2) \neq \emptyset$.}

 Suppose $C(K_1',K_2') \cap A(K_1,K_2) = \emptyset$. Then, since
 $S$ and $S'$ are not loosely non-crossing,
 $C(K_1,K_2) \cap A(K_1',K_2') \neq \emptyset$.
 By \eqref{CB'}, $C(K_1,K_2) \cap B(K_1',K_2') \neq \emptyset$.
 Let $D(K_1,K_2)$
 be a component of $A(K_1,K_2)$ such that $K_1 \cup K_2 \subseteq N(D(K_1,K_2))$.
 Since  $C(K_1',K_2') \cap A(K_1,K_2) = \emptyset$ it holds that
 either $D(K_1,K_2) \subseteq B(K_1',K_2')$ or
 $D(K_1,K_2) \subseteq A(K_1',K_2')$. In the former case
 $D(K_1,K_2)$ is anticomplete to $C(K_1,K_2) \cap A(K_1',K_2')$,
 and in the latter case  $D(K_1,K_2)$ is anticomplete to $C(K_1,K_2) \cap B(K_1',K_2')$; in both cases a contradiction.
 This proves \eqref{CA'}.
 \\
 \\
 By \eqref{CB'}, \eqref{CA'}, and symmetry each of the four sets
 $C(K_1,K_2) \cap A(K_1',K_2')$, $C(K_1,K_2) \cap B(K_1',K_2')$, $C(K_1',K_2') \cap A(K_1,K_2)$, $C(K_1',K_2') \cap B(K_1,K_2)$ is nonempty.
 Since each of the sets $K_1,K_2,K_1',K_2'$ is a clique,
 we may assume that  $K_1 \cap B(K_1',K_2') \neq \emptyset$,
 $K_2 \cap A(K_1',K_2')  \neq \emptyset$,
 $K_1' \cap  B(K_1,K_2)  \neq \emptyset$, and
 $K_2' \cap A(K_1,K_2) \neq \emptyset$, and therefore
  $K_1 \subseteq B(K_1',K_2') \cup K_1' \cup K_2'$,
 $K_2 \subseteq A(K_1',K_2') \cup K_1' \cup K_2'$, $K_1' \subseteq B(K_1,K_2) \cup K_1 \cup K_2$, and $K_2' \subseteq A(K_1,K_2) \cup K_1 \cup K_2$.

\sta{\label{hugecomp} There is a component $D$ of
      $A(K_1, K_2) \cup A(K_1',K_2')$ such that
     $K_1 \cup K_2 \cup K_1' \cup K_2' \subseteq N[D]$.}

  Let $D(K_1,K_2)$ be  a component of $A(K_1,K_2)$ such that
  $K_1 \cup K_2 \subseteq N(D(K_1,K_2))$ and let
 $D(K_1',K_2')$ be  a component of $A(K_1',K_2')$ such that
  $K_1' \cup K_2' \subseteq N(D(K_1',K_2'))$ (such components exist because $S$ and $S'$ are active,
  and hence proper).
  Since   $C(K_1,K_2) \cap B(K_1',K_2') \neq \emptyset$
  and $C(K_1,K_2) \cap A(K_1',K_2') \neq \emptyset$
  it follows that $D(K_1,K_2) \not \subseteq A(K_1',K_2')$
  and  $D(K_1,K_2) \not \subseteq B(K_1',K_2')$, and
  therefore $D(K_1,K_2) \cap C(K_1',K_2') \neq  \emptyset$.
  Similarly  $D(K_1',K_2')  \cap C(K_1,K_2) \neq  \emptyset$.
  Consequently $D(K_1,K_2) \cup D(K_1',K_2')$ is connected.
  Now set $D$ to be the component of
  $A(K_1, K_2) \cup A(K_1',K_2')$ that contains $D(K_1,K_2) \cup D(K_1',K_2')$,
      and \eqref{hugecomp} holds.
\\
\\
Since $(G,w)$ is $8$-unbalanced, there is a component $B$ of
 $G \setminus (K_1 \cup K_1' \cup K_2 \cup K_2')$ with $w(B) > \frac{1}{2}$.
 Then $B \subseteq B(K_1,K_2) \cap B(K_1',K_2')$. 
Let $C=N(B)$ and let $A=G \setminus (B \cup C)$. 
Then $(A,C,B)$ is a separation of $G$. Note that $C \subseteq (C(K_1,K_2) \cup C(K_1',K_2')) \setminus (A(K_1,K_2) \cup A(K_1',K_2'))$

\sta{\label{K2K2'} $K_2 \cap K_2' \neq \emptyset$ and $C \cap (K_1 \cup K_1')$ is not a clique.}

 Note first that, since $B \subseteq B(K_1, K_2)$, we have $N(B) \subseteq (C(K_1, K_2) \cup C(K_1', K_2')) \setminus A(K_1, K_2)$. Then in view of the last sentence before \eqref{hugecomp}, this means $N(B) \subseteq K_1 \cup K_2 \cup K_1'$. Similarly, since $B \subseteq B(K_1', K_2')$, we obtain that $N(B) \subseteq K_1' \cup K_2' \cup K_1$. 
 
 This shows that, if $K_2 \cap K_2' = \emptyset$, or if
 $C \cap (K_1 \cup K_1')$ is a clique, then $C$ is the union of two cliques,
 say $X$ and $Y$, and
 so $(A,C,B)$ is a two-clique-separation of $G$.
 We claim that $(A,C,B)$ is proper.
 By \eqref{hugecomp}
 there is a component $D$ of $A$  such that $K_1 \cup K_2 \cup K_1' \cup K_2' \subseteq N[D]$, and therefore  $C \subseteq N(D)$. If $|C|>2$, the claim follows.
 Since $G$ does not admit a clique cutset, we may assume that
 $X=\{x\}$ and $Y=\{y\}$ and $x$ is non-adjacent to $y$.
 We need to show that $A$ is not a path from $x$ to $y$.
 Suppose it is. Then every vertex of $A$ has exactly two neighbors in
 $A \cup X \cup Y$, and each of $x,y$ has exactly one neighbor in
 $A$.  Since $A(K_1,K_2) \cup A(K_1',K_2') \subseteq A$, this
 contradicts Lemma \ref{properdeg3}. This proves the claim that $(A,C,B)$ is proper. 
  
 Observe that  $B \cup C \subseteq B(K_1,K_2) \cup K_1 \cup K_2$.
 Since  $C(K_1,K_2) \cap A(K_1',K_2') \neq \emptyset$, the inclusion is proper and we
 get a contradiction to the fact that $S$ is active. This proves
 \eqref{K2K2'}.
 \\
 \\
 In view of \eqref{K2K2'}, we
 write $K_2 \cap K_2'=\{s\}$. Note that $|K_2|, |K_2'| = 2$, since we know $K_2 \cap A(K_1', K_2')$ and $K_2' \cap A(K_1, K_2)$ are non-empty, and $s \notin A(K_1, K_2) \cup A(K_1', K_2')$. Hence write $K_2=\{s,t\}$ and $K_2'=\{s,r\}$, with $t \in A(K_1',K_2')$ and $r \in A(K_1, K_2)$. Also by \eqref{K2K2'}, there 
 exist non-adjacent $k_1 \in K_1 \cap C$ and $k_1' \in K_1' \cap C$.
 Let $P$ be a path from $k_1$ to $k_1'$ with $P^* \subseteq B$.
 Let $Q$ be a path from $k_1$ to $k_1'$ with $Q^* \subseteq D$ where $D$ is as in \eqref{hugecomp}. 
 Then $H=k_1 \dd P \dd k_1' \dd Q \dd k_1$ is a hole.
 
\sta{\label{AA'}
$A(K_1,K_2) \cap A(K_1',K_2')\neq \emptyset$.}

 Suppose that  $A(K_1,K_2) \cap A(K_1',K_2')= \emptyset$.
Since $S'$ is proper, $r$ has a neighbor $x \in A(K_1',K_2')$. Since $r \in A(K_1,K_2)$, we have $x \in A(K_1,K_2) \cup C(K_1,K_2)$, but by assumption, $A(K_1,K_2) \cap A(K_1',K_2') = \emptyset$, so we conclude $x \in C(K_1,K_2) = K_1 \cup \{s, t\}$. From above, $K_1 \subseteq B(K_1',K_2') \cup C(K_1',K_2')$, and $s \in C(K_1',K_2')$, so the only possible neighbor of $r$ lying in $A(K_1', K_2')$ is $t$. 
 But now $\{s,t,r\}$ is a triangle, contrary to
 the fact that $G \in \mathcal{F}_2$.
 This proves \eqref{AA'}.
 \\
 \\
 Since $N(A(K_1,K_2) \cap A(K_1',K_2')) \subseteq K_2 \cup K_2' \cup (K_1 \cap K_1')$, and since $K_2 \cup K_2'$ is not a star cutest in $G$, it follows that
 $K_1 \cap K_1' \neq \emptyset$. Let $x \in K_1 \cap K_1'$.
 Now $x$ has two non-adjacent neighbors in $H$, namely $k_1$ and $k_1'$,
 contrary to the fact that $G \in \mathcal{F}_2$.
 \end{proof}

\section{Heavy seagulls}
\label{sec:heavyseagulls}

A {\em seagull} is a graph that is a three-vertex path.
Given a seagull $F=a \dd v \dd u$ in $G$, an induced subgraph 
$T$ of $G$ is a {\em theta through $F$} if $T$ is a theta, one of $a$, $u$ is
an end of $T$, and $F \subseteq T$.
A seagull $a \dd v \dd u$ is {\em heavy} if
$\deg_G(a)>2$ and $\deg_G(u)>2$.
A heavy seagull is {\em extendable} if there is a theta through it in $G$.
The goal of this section is to show that every heavy seagull is
``broken'' by some two-clique-separation.
We start with a lemma.

\begin{lemma}
\label{nopaththeta}
Let $G \in \mathcal{F}_2$, let $F=a$-$v_1$-$u_1$ be  a seagull in $G$ and let $T$ be a theta
through $F$ in $G$. Let the ends of $T$ be $a,b$ and let the paths of $T$ be 
$P_1,P_2,P_3$  where $F \subseteq P_1$. Assume that $T$ is chosen with
$|P_1|$ minimum among all thetas through $F$ with end $a$ in $G$.
Let $P$ be a path from $u_1$ to $(P_2 \cup P_3) \setminus N[b]$.
Then $P^*$ contains a vertex of $N[b] \cup N[v_1]$.
\end{lemma} 

\begin{proof}
Suppose for a contradiction that 
$P^* \cap (N[b] \cup N[v_1])=\emptyset$. Let $N_T(b)=\{w_1,w_2,w_3\}$ where
$w_i \in P_i$. 
Then $P$ contains a path $Q=q_1 \dd \dots \dd q_k$ such that
$q_1$ has a neighbor in $P_1 \setminus \{a,v_1,b\}$,
$q_k$ has a neighbor in $(P_2 \cup P_3) \setminus \{b,w_2,w_3\}$ 
and $Q \cap T = \emptyset$. We may assume that $Q$ is chosen in such a 
way that $k$ is minimum. We may also assume that $q_k$ has a neighbor $s$ in
$P_2 \setminus \{b,w_2\}$. Since $G \in \mathcal{F}_2$, it follows that 
$N_T(q_k)=\{s\}$.
Let $t$ be a neighbor of $q_1$ in $P_1^* \setminus \{v_1\}$; similarly 
$N_T(q_1)=\{t\}$.
In particular $k>1$.
It follows from the minimality of $k$ that
$Q^*$ is anticomplete to $T \setminus \{w_2,w_3\}$. Moreover, since $s$-$Q$-$t$-$P_1$-$a$-$P_2$-$s$ is a hole, it follows that each of $w_2, w_3$ has at most one neighbor in $Q$.

\sta{\label{notboth} Not both $w_2$ and $w_3$ have a neighbor in $Q$.}

Suppose not. Let $i, j \in \{1, \dots, k\}$ be such that $q_i$ is adjacent to $w_3$ and $q_j$ is adjacent to $w_2$. Since $N_T(q_k) = \{s\}$, it follows that $i, j \neq k$. Now, $w_3$-$P_3$-$a$-$P_2$-$w_2$-$q_j$-$Q$-$q_i$-$w_3$ is a hole, and $b$ has two neighbors in it, 
contrary to the fact that $G \in \mathcal{F}_2$. This proves \eqref{notboth}. 

\sta{\label{w3} $w_3$ is anticomplete to $Q$.}

Suppose not. Let $i \in \{1, \dots, k\}$ be such that $q_i$ is adjacent to $w_3$. Then, by \eqref{notboth}, it follows that $w_2$ has no neighbor in $Q$, and so  $s$-$P_2$-$b$-$P_1$-$t$-$Q$-$s$ is a hole and $w_3$ has two neighbors $b$ and $q_i$ in it, contrary to the fact that $G \in \mathcal{F}_2$. This proves \eqref{w3}. 

\sta{\label{w2} $w_2$ is anticomplete to $Q$.}

Suppose $w_2$ has a neighbor in $Q$; let
$i \in \{1, \dots k\}$ be such that $w_2$ is adjacent to $q_i$. Let $S$ be the  path
$w_1 \dd P_1 \dd t \dd q_1 \dd Q \dd q_k$.
Since $t \neq v_1$, we have that $v_1 \not \in S$.
Now  $H=b \dd w_1 \dd S \dd q_k \dd s \dd P_2 \dd a \dd P_3 \dd b$ is a hole
and $b,q_i \in N_H(w_2)$, contrary to the fact that $G \in \mathcal{F}_2$.
This proves \eqref{w2}.
\\
\\
Since $s \neq w_2$ and $t \neq v_1$ the paths $t \dd P_1 \dd a$,
$t \dd q_1 \dd Q \dd q_k \dd s \dd P_2 \dd a$ and
$t \dd P_1 \dd b \dd P_3 \dd a$  form a theta through  $\{a,v_1,u_1\}$
that contradicts the choice of $T$ with $|P_1|$ minimum.
\end{proof}

The next result allows us to use Lemma \ref{nopaththeta} to handle heavy seagulls.

\begin{lemma}
  \label{norestricted}
  Let $G \in \mathcal{F}_2$ and let $F$ be a heavy seagull in $G$.
  Assume that $G$ does not admit a star cutset.
  Then $F$ is extendable.
\end{lemma}

\begin{proof}
  Let $F=a \dd v \dd u$. Since $F$ is heavy, there exist $x_1,x_2 \in N(a) \setminus \{v\}$. Since $G \in \mathcal{F}_2$  the set $\{x_1,v,x_2\}$ is stable. Since
  $G$ does not admit a star cutset, it follows that for $i \in \{1,2\}$, there exists a path $P_i$ from $x_i$ to
  $u$ with $P_i^* \cap N[a]=\emptyset$. By choosing $P_1,P_2$ with $P_1 \cup P_2$
  minimal, and permuting the indices if necessary, we may assume that one of the following two cases holds.
  \begin{enumerate}
  \item $P_1^* \subseteq P_2^*$ and $x_1$ has a neighbor in $P_2^*$.
  \item There exists a vertex $q \in V(G) \setminus \{v,a,x_1,x_2\}$ and
    a path $Q$ from $u$ to $q$ such that $P_i=u \dd Q \dd q \dd P_i' \dd x_i$
    and $P_1' \setminus q$ is disjoint from and anticomplete to $P_2' \setminus q$.
  \end{enumerate}

  We handle the former case first. Let $P_2=p_1 \dd \dots \dd p_k$
  where $p_1=u$ and $p_k=x_2$.  Let $i$ be maximum such that both $x_1$ and
  $v$ have neighbors in $p_i \dd P_2 \dd p_k$. Then there exists
  $x \in \{x_1,v\}$ such that $x$ is anticomplete to $\{p_{i+1}, \dots, p_k\}$,
  and consequently $H = x \dd p_i \dd P_2 \dd p_k \dd a \dd x$ is a hole.
  Let $y \in \{x_1,v\} \setminus \{x\}$.  Since $y$ is adjacent to $a$
  and has a neighbor in $\{p_i, \dots, p_k\}$, if follows that $y$
  has at least two neighbors in $H$, contrary to the fact that
  $G \in \mathcal{F}_2$. This proves that the first case is impossible, and so the second case holds.
    Now let $H'$ be the hole $q \dd P_2' \dd x_2 \dd a \dd x_1 \dd P_1' \dd q$.
    Since $v$ is adjacent to $a$ and $G \in \mathcal{F}_2$, it follows that
    $v$ is anticomplete to $P_1' \cup P_2'$, and in particular, $u\not\in V(H')$. 
    Let $R$ be a shortest path from $u$ to a vertex $u'$ with a neighbor in $H'$ such that $R$ is contained in $G \setminus (N[v] \setminus \{a,u\})$. Such a path exists, since $v$ is not a star cutset center. Since $G \in \mathcal{F}_2$, it follows that $u'$ has a unique neighbor $h$ in $H'$. If $h \not\in \{x_1, x_2, a\}$, then $H' \cup R \cup \{v\}$ is a theta in $G$ with ends $h$ and $a$, and paths $a$-$v$-$u$-$R$-$u'$-$h$ and the  the two paths from $h$ to $a$ in $H'$, and the result holds. So (by symmetry) we may assume that $h \in \{x_1,a\}$. 
    
    Let $R’$ be the path from $h$ to $q$ with interior in $R \cup Q$.
Write $R'=r_1 \dd \dots \dd r_t$, where $r_1=h$, $r_t=q$, and there exists  $i \in \{2, \dots, t-1\}$ 
such that $r_1, \dots, r_i \in R$  and $r_{i+1}, \dots, r_t \in Q$.
Suppose first that $v$ has a neighbor $w $ in $\{r_{i+1},..r_t\}$. Then 
$h \dd R' \dd q \dd P_2' \dd x_2 \dd a \dd h$ is a hole,
and $v$ has two neighbors  in it (namely $a$ and $w$), contrary to the fact that $G \in \mathcal{F}_2$.
So $v$ is anticomplete to $\{r_{i+1}, \dots, r_t\}$.

If  $v$ is anticomplete to $Q \setminus u$, then  $H’ \cup Q \cup \{v\}$ is a theta with ends $a,q$ and paths $a \dd v \dd u \dd Q \dd q$  and the  the two paths from $a$ to $q$ in $H'$, and so
$F$ is extendable. Thus we may assume that $v$ has a neighbor  in $Q \setminus u$, and therefore  $u$ is distinct from and  non-adjacent to $r_{i+1}$.

Next suppose that  $r_i$ is adjacent to $a$.  Then $i=2$ and   $h=a$.  Let $Q’$ be the path
from $a$ to $q$ contained in $Q \cup \{a,v\}$ (thus $Q'$ is obtained from   $a \dd v \dd u \dd Q \dd q$ by shortcutting through an edge incident with $v$).
Then $a,r_{i+1} \in Q'$. Now $a \dd Q'  \dd q \dd P_2' \dd x_2 \dd a$ is a hole, and $r_i$ has two 
neighbors in it (namely $a$ and $r_{i+1}$), contrary to the fact that $ G \in \mathcal{F}_2$.    This proves  
that $r_i$ is non-adjacent to $a$.

Now there is a path $S$ from $u$ to $q$ with $S \subseteq u \dd R \dd r_i  \cup  r_{i+1} \dd Q \dd q$.
It follows that $\{a,v\}$  is anticomplete to $S \setminus u$.
Consequently, $ a \dd v \dd u \dd S$ is a path from $a$ to $q$. If $x_1$ has a neighbor $s \in  S$, then
$x_1$ has two neighbors in the hole $a \dd S \dd q \dd P_2'  \dd x_2 \dd a$ (namely $a$ and $s$), 
contrary to the fact that $G \in \mathcal{F}_2$. This proves that $x_1$ is anticomplete to $S$.
But now $H’ \cup S$ is a theta with ends $a,q$ and paths $S$  and the two paths from $a$ to $q$ in $H'$, and so $F$ is extendable. 
 
 \end{proof}

Now we  deal with extendable seagulls.

\begin{theorem}
  \label{breakextheavyseagull}
  Let $G \in \mathcal{F}_2$ and let $(G,w)$ be a $4$-unbalanced pair.
  Assume that $G$ does not admit a star cutset. Let $F=a \dd v_1 \dd u_1$ be a
  heavy seagull in $G$.  Then there are two cliques $K_1,K_2$ of $G$
  such that $S(K_1,K_2)$ is active and
  $A(K_1,K_2) \cap \{a,u_1\} \neq \emptyset$.
    \end{theorem}

\begin{proof}
  Let $T$ be a theta through $F$ (such $T$ exists by Lemma~\ref{norestricted}). We may assume that $a$ is an end of $T$;
  let the other end be $b$. Let the paths of $T$ be $P_1,P_2,P_3$ with
  $v_1 \in P_1$, and $T$ is   chosen with
$|P_1|$  minimum among all thetas through $F$ in $G$ with end $a$.

\sta{\label{breakau}  Let $D$ be a component of
    $G \setminus ((N[b] \setminus N[v_1]) \setminus \{a,u_1\})$. Then
    $|D \cap \{a,u_1\}| \leq 1$.}
 
Since $G \in \mathcal{F}_2$, we have that $|V(P_i)| \geq 4$ 
and so $v_1,u_1 \in P_1 \setminus \{b\}$.
Suppose for a contradiction that   $u_1, a \in D$. Then there is a path $P$
from $u_1$ to $a$ with $P^* \subseteq D$.  Consequently
$P^*$ contains no vertex of $N[b] \cup N[v_1]$.
Since $a \in (P_2 \cup P_3) \setminus N[b]$ 
we get a contradiction to Lemma \ref{nopaththeta} applied to $F,T$ and $P$.
This proves \eqref{breakau}.

\sta{\label{XvXb} There are cliques
    $X,Y$ of $G$ and a separation $(A, X \cup Y, B)$
    such that $a \in A$ and $u_1 \in B$.}

Let $D_a,D_u$ be the components of $G \setminus ((N[b] \cup N[v_1]) \setminus \{a,u_1\})$ with $a \in D_a$ and $u_1 \in D_u$. By \eqref{breakau}, we have that
$D_a \neq D_u$. It follows that there is a separation
$S=(A, (N[b] \cup N[v_1]) \setminus \{a,u_1\},B)$ of $G$ with
$D_a \subseteq A$ and $D_u \subseteq B$. Now \eqref{XvXb} follows
from Lemma \ref{startoclique} applied to $S$. This proves \eqref{XvXb}.
\\
\\
Let $X,Y$ be as in \eqref{XvXb}. 
Since $G \in \mathcal{F}_2$ and since  $(G,w)$ is a $4$-unbalanced pair,
the canonical two-clique-separation corresponding to $\{X,Y\}$ is defined,
and by \eqref{breakau} $|B(X,Y) \cap \{a,u_1\}| \leq 1$.
Since $|B(X,Y) \cap \{a,u_1\}| \leq 1$, we deduce that 
$A(X,Y) \cap \{a,u_1\} \neq \emptyset$; let
$p \in A(X,Y) \cap \{a,u_1\}$. Let $D$ be the component of $A(X,Y)$ containing
$p$, and let $N=N(D)$. Then $N$ is the union of two cliques $K_1,K_2$.

\sta{\label{K1K2proper} The pair $\{K_1,K_2\}$ is
    proper.}

  Observe that $B(X,Y) \subseteq B(K_1,K_2)$ and $D \subseteq A(K_1,K_2)$.
  Since $G$ does not admit a clique cutset, both $K_1$ and $K_2$ are non-empty.
  If  $|K_1 \cup K_2| \geq 3$, then
 $D$ is a component of
$A(K_1,K_2)$ with $K_1 \cup K_2 \subseteq N(D)$, and the claim holds.
Thus we may assume that $|K_1|=|K_2|=1$.
Since $F$ is heavy, it follows that
$\deg_G(p)>2$, and therefore $D \cup K_1 \cup K_2$ is not a path from
$K_1$ to $K_2$, and again the claim holds. This proves \eqref{K1K2proper}.
\\
\\
Now among all proper pairs $(K_1',K_2')$ with
$B(K_1', K_2') \cup K_1' \cup K_2' \subseteq B(K_1, K_2) \cup K_1 \cup K_2$
choose $K_1',K_2'$ with $B(K_1',K_2') \cup K_1' \cup K_2'$ inclusion-wise
minimal, and subject to that with $B(K_1',K_2')$ inclusion-wise maximal.
Then $(K_1',K_2')$ is active and $A(K_1',K_2') \cap \{a,u_1\} \neq \emptyset$.
This proves Theorem \ref{breakextheavyseagull}.
\end{proof}

\section{Proof of Theorem \ref{t=2}}
\label{sec:trianglefree}

We begin with proving an extension of Theorem \ref{boundeddeg}. For a graph $G$ and positive integer $d$, we denote by $\gamma_d(G)$ the maximum
degree of the subgraph of $G$ induced by the
set of vertices with degree at least $d$ in $G$.

\begin{theorem}
  \label{smalldensecomps}
  For all $k, \gamma > 0$, there exists $w = w(k, \gamma)$ such that every
graph $G$ with $\gamma_3(G) \leq \gamma$
and treewidth more than $w$ contains a subdivision of $W_{k\times k}$
or the line graph of a subdivision of $W_{k\times k}$.
\end{theorem}
\begin{proof}
 Let $w=w(k,\gamma)=f(c(k,\gamma+3))$, where $f$ is as in Theorem \ref{wallminor} and $c$ is an in Theorem \ref{boundeddeg}. Let $G$ be a graph with treewidth at least $w$. By Theorem \ref{wallminor}, $G$ has a subgraph $X$ which is isomorphic to $W_{c(k,\gamma+3)\times c(k,\gamma+3)}$. Let $H=G[V(X)]$. Then $H$ has treewidth at least $c(k,\gamma+3)$. Also, we claim that $G$ has maximum degree at most $\gamma+3$. To see this, suppose for a contradiction that $H$ has a vertex $v$ of degree at least $\gamma+4>3$. Then, since $X$ has maximum degree at most $3$, there are at least $\gamma+1$ edges in $E(H)\setminus E(X)$ incident with $v$. Moreover, for each such edge, its end distinct from $v$ has degree at least two in $X$, and so degree at least $3$ in $H$. But then $v$ is a vertex of degree at least $3$ in $G$ with at least $\gamma+1$ neighbors, each of degree at least $3$ in $G$. This violates $\gamma_3(G)\leq \gamma$, and so proves the claim. Now, by Theorem \ref{boundeddeg}, $H$, and so $G$, contains a subdivision of $W_{k\times k}$
or the line graph of a subdivision of $W_{k\times k}$. 
\end{proof}
We remark that Theorem \ref{smalldensecomps} is sharp, in the sense that the conclusion fails if the number $3$ in $\gamma_3(G)$ is replaced by any larger integer. This is due to the construction of \cite{Davies}, in which the set of vertices of degree $4$ or more is stable. Next, we deduce:
\begin{theorem}
  \label{nobarbell}
  For all $t$, there exists $M = M(t)$ such that every graph in
  $\mathcal{F}_t$ with no heavy seagull  
    and with treewidth more than $M$ contains a subdivision of $W_{t\times t}$.
\end{theorem}

\begin{proof}
     Since $G$ contains no heavy seagull, it follows that no two vertices of degree at least three in $G$ are at distance two in $G$. This implies that every connected component of the subgraph of $G$ induced by the set of vertices of  degree at least three in $G$ is a clique, and therefore has size at most $t$. It follows that $\gamma_3(G)\leq t-1$. Also, since $G \in \mathcal{F}$, no induced subgraph of
  $G$ is the line graph of a  subdivision of $W_{3\times 3}$. Now
  Theorem \ref{nobarbell} follows from Theorem \ref{smalldensecomps}.
  \end{proof}

We are now ready to prove Theorem \ref{t=2}, the main result of this section, which we restate. 
\begin{theorem}
  \label{t=2_restate}
  For all $k$, there exists $c = c(k)$ such that every graph
in $\mathcal{F}_2$ with no star cutset
    and with treewidth more than $c$ contains a subdivision of $W_{k\times k}$.
\end{theorem}
\begin{proof}
Let $M=M(k) \geq 1$ be as in Theorem \ref{nobarbell}.
Let $G \in \mathcal{F}_2$ and assume that $G$ does not contain a subdivision of $W_{k\times k}$. We show that $\tw(G) \leq 8(M+1)$. Suppose not.
By Lemma \ref{lemma:bs-to-tw}, there is a weight function $w$ on $G$ such that $(G,w)$ is $4(M+1)$-unbalanced, and in particular
$8$-unbalanced. 
  Let $\mathcal{H}$ be the  set of all 
  heavy seagulls of
  $G$. By Lemma \ref{norestricted}, every seagull in $\mathcal{H}$ is extendable.
 Let $\mathcal{S}$ be the set of all pairs of cliques $\{K_1,K_2\}$
  obtained by applying Theorem \ref{breakextheavyseagull} to each member of $\mathcal{H}$.
  Then all elements of $\mathcal{S}$ are active. Let $\mathcal{T}$ be the set of the canonical two-clique-separations
  corresponding to the members of $\mathcal{S}$. By Theorem \ref{noncrossing2cliques}
  every pair of members of $\mathcal{T}$ is loosely non-crossing. Let
  $\beta$ be a central bag for $\mathcal{T}$.

 \sta{\label{simplerbeta} There is no heavy seagull in $\beta$.}

  Suppose $X=a \dd b \dd c$ is a heavy seagull in $\beta$. Then $X \in \mathcal{H}$, and so there is a separation $(A,B,C) \in \mathcal{T}$ such that $\{a,c\}\cap   A \neq \emptyset$. We may assume that $a \in A$. It follows from the
  definition of $\beta$ that there exists a pair $\{K_1,K_2\} \in \mathcal{S}$
  such that $a \in P_{K_1K_2}^*$. But then $\deg_{\beta}(a)=2$, contrary to the fact that $X$ is a heavy seagull of $\beta$. This proves \eqref{simplerbeta}.
\\
\\
Recall that for  $v \in V(G)$ we have defined 
$\delta_{\mathcal{S}}(v)=\bigcup_{K \text {: }v \in K \text{ and there exists } L \text{ such that }\{K,L\} \in \mathcal{S}}K.$

 \sta{\label{smalldelta} $|\delta_{\mathcal{S}}(v)| \leq 2$
      for every $v \in \beta$.}

  Suppose $|\delta_{\mathcal{S}}(v)| >2$ for some $v \in \beta$. Then there exist pairs
  $\{K_1,K_2\}, \{K_1', K_2'\} \in \mathcal{S}$ such that $v \in K_1 \cap K_1'$.
  Let $K_1=\{k_1,v\}$ and $K_1'=\{k_1',v\}$. Since $G \in \mathcal{F}_2$, it follows that $k_1 \dd v \dd k_2$ is a seagull in $G$. Since $k_1 \in K_1$,
  it follows from Lemma $\ref{active}$ that $k_1$ has a neighbor in $B(K_1,K_2)$.
  Since all elements of $\mathcal{S}$ are active, and therefore proper, we deduce that $k_1$ has a 
  neighbor in $A(K_1,K_2)$. Since $v \in C(K_1,K_2)$, we deduce that $\deg_G(k_1)>2$. 
  Similarly, $\deg_G(k_1')>2$. Consequently, $k_1 \dd v \dd k_1'$ is a heavy
  seagull of $G$. It follows that there exists a pair $\{L_1,L_2\} \in \mathcal{T}$
  such that $A(L_1,L_2) \cap \{k_1,k_1'\} \neq \emptyset$, say $k_1 \in A(L_1,L_2)$. But then $k_1 \in A(L_1,L_2) \cap C(K_1,K_2)$, contrary to Theorem
  \ref{noncrossing2cliques}. This proves \eqref{smalldelta}.
  \\
  \\
  It follows from \eqref{simplerbeta} that there is no heavy  seagull in
  $\beta$.  By Theorem \ref{nobarbell}, since $G$ does not contain a subdivision of $W_{k\times k}$, 
  we have that $\tw(\beta) \leq M$.
    Let   $w_\beta$ be the inherited weight function on $\beta$.
 Since $\tw(\beta) \leq M$,  Lemma \ref{lemma:tw-to-weighted-separator} implies that
  $(\beta,w_\beta)$
 is $(M+1)$-balanced. Now, by \eqref{smalldelta} and  Theorem~\ref{unbalancedbeta} $(G,w)$ is
 $\max(4(M+1),2(M+1))$-balanced, and therefore $(G,w)$ is $4(M+1)$-balanced, a contradiction. 
\end{proof}

\section{Putting everything together}
\label{sec:theend}

In this section, we prove Theorem \ref{sparsethm}, which we restate.

\begin{theorem}
  \label{sparsethm_restate}
  For all $t> 0$, there exists $c = c(t)$ such that every graph in
  $\mathcal{F}_t$ with 
  treewidth more than $c$ contains a subdivision of $W_{t\times t}$ as an
  induced subgraph.
\end{theorem}

\begin{proof}
  Let $c = c(t)$ be as in Theorem \ref{t=2_restate}. By increasing $c(t)$, we may assume that $c(t) \geq t$.
  Let $G \in \mathcal{F}_t$, and suppose that $\tw(G)>c$.
  Lemma 7 from \cite{cliquetw} shows that clique cutsets do not affect treewidth, and so we may assume that $G$ does not admit a clique cutset.  Now we deduce from Lemma \ref{startoclique} that $G$ does not admit a star cutset.
  By Lemma \ref{cliqueortrianglefree} it follows that either $G \in \mathcal{F}_2$, or $G$ is a complete graph (and so $\tw(G) \leq t$). So we may assume that $G \in \mathcal{F}_2$.   But now the result follows from Theorem \ref{t=2_restate}.
\end{proof}

\end{document}